\documentclass[12pt,reqno]{amsart}
\usepackage{geometry,epsfig,color}

\def\rmA{{\rm A}}
\def\rmB{{\rm B}}
\def\rmC{{\rm C}}
\def\I{{\rm I}}
\def\J{{\rm J}}

\def\U{{\rm U}}

\def\V{\mathcal{V}}
\def\DD{\,{\scriptstyle \Delta}\,}

\newcommand{\JS}[2]{\ensuremath{\mathrm{JS}_{#1}^{#2}}}

\newtheorem{theorem}{Theorem}[section]   
\newtheorem{corollary}[theorem]{Corollary}     
\newtheorem{lemma}[theorem]{Lemma}         
\newtheorem{proposition}[theorem]{Proposition}  
\newenvironment{notation}{{\bf Notation.\ }\rm}{\bigskip}
  
\newenvironment{remark}{{\bf Remark.\ }\rm}{\bigskip}  
\newenvironment{example}{\bf Example. \rm}{\bigskip}
\newenvironment{definition}{\bf Definition.\ \rm}{\bigskip}

\title[Permanents and Generalized Derangements]{Zeons, Permanents, the Johnson scheme, and Generalized Derangements}

\author{Philip Feinsilver and John McSorley}
\address{Department of Mathematics\hfill\break
Southern Illinois University \hfill\break
Carbondale, IL. 62901, U.S.A.}
\email{pfeinsil@math.siu.edu \\ jmcsorley@math.siu.edu}

\DeclareMathOperator{\per}{per\,}
\DeclareMathOperator{\tr}{tr\,}

\begin{document}

\begin{abstract}
Starting with the zero-square ``zeon algebra" the connection with permanents is shown. Permanents of sub-matrices of a linear combination
of the identity matrix and all-ones matrix leads to moment polynomials with respect to  the exponential distribution.
A permanent trace formula analogous to MacMahon's Master Theorem is presented and applied.
Connections with permutation groups acting on sets and the Johnson association scheme arise. The families of numbers
appearing as matrix entries turn out to be related to interesting variations on derangements. These generalized
derangements are considered in detail as an illustration of the theory.
\end{abstract}

\maketitle

\thispagestyle{empty}

\section{Introduction}
Functions acting on a finite set can be conveniently expressed using matrices, whereby the composition of functions
corresponds to multiplication of the matrices. Essentially one is considering the induced action on the vector space
with the elements of the set acting as a basis. This action extends to tensor powers of the vector space. One can take
symmetric powers, antisymmetric powers, etc., that yield representations of the multiplicative semigroup of functions.
 An especially interesting representation occurs by taking non-reflexive, symmetric powers. Identifying the underlying set
of cardinality $n$ with $\{1,2,\ldots,n\}$, the vector space has basis $e_1,e_2,\ldots$. The action we are interested in
may be found by saying that the elements $e_i$ generate a ``zeon algebra", the relations being that the $e_i$ commute,
with $e_i^2=0$, $1\le i\le n$. To get a feeling for this, first we recall the action on Grassmann algebra where the matrix elements of the induced action
arise as determinants. For the zeon case, permanents appear. \bigskip

An interesting connection with the centralizer algebra of the action of the symmetric group comes up.
For the defining action on the set $\{1,\ldots,n\}$, represented as $0$\,-$1$ permutation matrices, the centralizer algebra
of $n\times n$ matrices commuting with the entire group is generated by $I$, the identity matrix, and $J$, the all-ones matrix.
The question was if they would help determine the centralizer algebra for the action on subsets of a fixed size, $\ell$-sets, for $\ell>1$.
It is known that the basis for the centralizer algebra is given by the adjacency matrices of the Johnson scheme. Could one
find this working solely with $I$ and $J$? The result is that by computing the ``zeon powers", i.e., the action of $sI+tJ$, linear
combinations of $I$ and $J$,  on $\ell$-sets, the Johnson scheme appears naturally. The coefficients are polynomials in $s$ and $t$
occurring as moments of the exponential distribution. And they turn out to count derangements, and related generalized derangements.
The occurrence of Laguerre polynomials in the combinatorics of derangements is well-known. Here the ${}_2F_1$ hypergeometric
function, which is closely related to Poisson-Charlier polynomials, arises rather naturally.\bigskip

Here is an outline of the paper. \S2 introduces zeons and permanents. The trace formula is proved. Connections with the centralizer
algebra of the action  of the symmetric group on sets is detailed. \S3 is a study of exponential polynomials needed for the
remainder of the paper. Zeon powers of $sI+tJ$ are found in \S4 where the spectra of the matrices are found via the Johnson scheme.
\S5 presents a combinatorial approach to the zeon powers of $sI+tJ$, including an interpretation of exponential moment polynomials
by elementary subgraphs. In \S6, generalized derangement numbers, specifically counting derangements and counting arrangements are considered in detail. 
The Appendix has some derangement numbers and arrangement numbers
for reference, as well as a page of exponential polynomials. An example expressing exponential polynomials in terms of elementary subgraphs is given there.

\section{Representations of functions acting on sets}

Let $\V$ denote the vector space $\mathbb{Q}^n$ or $\mathbb{R}^n$. We will look at the
action of a linear map on $\V$ extended to quotients of tensor powers $\V^{\otimes\ell}$.
We work with coordinates rather than vectors. First, recall the Grassmann case.
To find the action on $\V^{\wedge\ell}$ consider an algebra generated by $n$ variables $e_i$ 
satisfying $e_ie_j=-e_je_i$. In particular, $e_i^2=0$. \bigskip

\begin{notation}   The standard $n$-set $\{1,\ldots,n\}$ will be denoted $[n]$. Roman caps
$\I$, $\J$, $\rmA$, etc. denote subsets of $[n]$. We will identify them with the corresponding
ordered tuples.  Generally, given an $n$-tuple $(x_1,\ldots,x_n)$ and a subset $\I\subset [n]$, we denote
products
$$x_\I=\prod_{j\in\I} x_j$$
where the indices are in increasing order if the variables are not assumed to commute. \par
As an index we will use $\U$ to denote the full set $[n]$.\par
Italic $I$ and $J$ will denote the identity matrix and all-ones matrix respectively. \bigskip

For a matrix $X_{\I\J}$, say, where the labels are subsets of fixed size $\l$, dictionary ordering is used.
That is, convert to ordered tuples and use dictionary ordering. For example, for $n=4$, $\l=2$, we have
labels $12, 13, 14, 23, 24, 34$ for rows one through six respectively.
\end{notation}

A basis for $\V^{\wedge\ell}$ is given by products 
$$e_\I=e_{i_1}e_{i_2}\cdots e_{i_\ell}$$
$\I\subset [n]$, where we consider $\I$ as an ordered $\ell$-tuple.
Given a matrix $X$ acting on $\V$, let
$$y_i=\sum_j X_{ij}e_j$$
with corresponding products $y_\I$.
Then the matrix $X^{\wedge \ell}$ has entries given by the coefficients in the expansion
$$y_\I=\sum_{\J} (X^{\wedge \ell})_{\I\J} e_\J$$
where the anticommutation rules are used to order the factors in $e_\J$.
Note that the coefficient of $e_j$ in $y_i$ is $X_{ij}$ itself. And for $n>3$, 
the coefficient of $e_{34}$ in $y_{12}$ is
$$\det\begin{pmatrix} X_{13}&X_{14}\\X_{23}&X_{24}\end{pmatrix}$$
We see that in general the $\I\J$ entry  of $X^{\wedge \ell}$ is the minor of $X$
with row labels $\I$ and column labels $\J$. A standard term for the matrix
$X^{\wedge \ell}$ is a \textsl{compound matrix}. Noting that $X^{\wedge \ell}$ is
$\binom{n}{\ell}\times\binom{n}{\ell}$, in particular $\ell=n$ yields the one-by-one matrix
with entry equal to $(X^{\wedge n})_{\U\U}=\det X$. 
\bigskip

In this work, we will use the algebra of \textsl{zeons}, standing for ``zero-ons", or more specifically,
``zero-square bosons". That is, we assume that the variables $e_i$ satisfy the properties
$$e_ie_j=e_je_i \qquad \text{and} \qquad e_i^2=0$$
A basis for the algebra is again given by $e_\I$, $\I\subset [n]$. At level $\ell$, the induced matrix
$X^{\vee\ell}$ has $\I\J$ entries according to the expansion of $y_\I$
$$y_i=\sum_j X_{ij}e_j \qquad \longrightarrow \qquad  y_\I=\sum_{\J} (X^{\vee \ell})_{\I\J} e_\J$$
similar to the Grassmann case. Since the variables commute, we see that the $\I\J$ entry
of $X^{\vee\ell}$ is the \textsl{permanent} of the submatrix with rows $\I$ and columns $\J$. In particular,
$(X^{\vee n})_{\U\U}=\per X$. We refer to the matrix $X^{\vee\ell}$ as the ``$\ell^{\rm th}$ zeon power of $X$".  \bigskip

\begin{subsection}{Functions on the power set of $[n]$}
Note that $X^{\vee\ell}$ is indexed by $\ell$-sets. Suppose $X_f$ represents a function $f\colon [n]\to[n]$.
So it is a zero-one matrix with $(X_f)_{ij}=1$ the single entry in row $i$ if $f$ maps $i$ to $j$. The 
$\ell^{\rm th}$ zeon power of $X$ is the matrix of the induced map on $\ell$-sets. If $f$ maps an $\ell$-set 
$\I$ to one of lower cardinality, then the corresponding row in $X^{\vee\ell}$ has all zero entries. Thus,
the induced matrices in general correspond to ``partial functions". \bigskip

However, if $X$ is a permutation matrix, then $X^{\vee\ell}$ is a permutation matrix for all $0\le\ell\le n$.
So, given a group of permutation matrices, the map $X\to X^{\vee\ell}$ is a representation of the group.
\end{subsection}

\begin{subsection}{Zeon powers of $sI+tX$}
Our main theorem computes the $\ell^{\rm th}$ zeon power of $sI+tX$ for an  $n\times n$ matrix $X$, where $s$ and $t$ are
scalar variables.

\begin{theorem} \label{thm:zpowers}
For a given matrix $X$, for $0\le\ell\le n$, and indices $|\I|=|\J|=\ell$,
$$\left((sI+tX)^{\vee\ell}\right)_{\I\J}
=\sum_{0\le j\le\ell} s^{\ell-j}t^j\,\sum_{\substack{A\subset \I\cap\J\\|\rmA|=\ell-j}}\left(X^{\vee j}\right)_{\I \setminus \rmA,\J \setminus \rmA}$$
\end{theorem}

\begin{figure}[!ht]
\begin{center}
\scalebox{.5}{\input{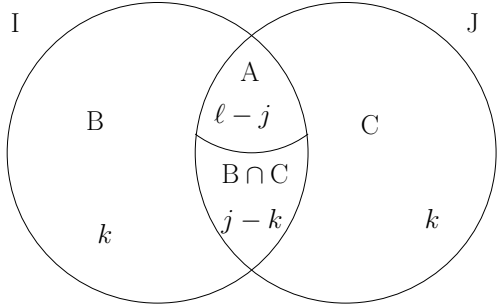}}
\caption{Configuration of sets. $\rmB=\I\setminus\rmA$, $\rmC=\J\setminus\rmA$ .}\label{fig:perms}
\end{center}
\end{figure}

\begin{proof} Start with $y_i=se_i+t\xi_i$, where $\xi_i=\sum_j X_{ij}e_j$.
Given $\I=(i_1,\ldots,i_\ell)$, we want the coefficient of $e_\J$ in the expansion of the product
$y_\I=y_{i_1}\cdots y_{i_\ell}$. Now,
$$y_\I=(se_{i_1}+t\xi_{i_1})\cdots(se_{i_\ell}+t\xi_{i_\ell})$$
Choose $\rmA\subset\I$ with $|\rmA|=\ell-j$, $0\le j\le \ell$. A typical term of the product has the form
$$s^{\ell-j}t^j e_\rmA \xi_\rmB$$
where $\rmA\cap\rmB=\emptyset$, $\rmB=\I\setminus \rmA$. $\xi_\rmB$ denotes the product of terms $\xi_i$ with indices in $\rmB$.
Expanding, we have
$$\xi_\rmB=\sum_C \left(X^{\vee j}\right)_{\rmB\rmC}\,e_\rmC$$
and
$$e_A\xi_\rmB=\sum_C \left(X^{\vee j}\right)_{\rmB\rmC}\,e_\rmA e_\rmC$$
Thus, for a contribution to the coefficient of $e_{\J}$, we have 
$\rmA\cup\rmC=\J$, where $\rmA\cap\rmC=\emptyset$. I.e., $\rmC=\J\setminus \rmA$ and $\rmA\subset\I\cap\J$. 
So the coefficient of $s^{\ell-j}t^j$ is as stated.
\end{proof}
\end{subsection}

\begin{subsection}{Trace formula}
Another main feature is the \textsl{trace formula} which shows the permanent of $I+tX$ as the generating function for the
traces of the zeon powers of $X$. This is the zeon analog of the theorem of MacMahon for representations on symmetric tensors.
\begin{theorem} \label{thm:ztrace}
We have the formula
$$\per (sI+tX)=\sum_j s^{n-j}t^j \tr X^{\vee j}$$
\end{theorem}
\begin{proof}
The permanent of $sI+tX$ is the $\U\U$ entry of $(sI+tX)^{\vee n}$. Specialize $\I=\J=\U$ in Theorem \ref{thm:zpowers}.
So $\rmA$ is any $(n-j)$-set with $\I\setminus A=\J\setminus A=A'$, its complement in $[n]$. Thus
\begin{align*}
\per (sI+tX)&=\left((sI+tX)^{\vee n}\right)_{\U\U} \\ 
&=\sum_{0\le j\le n} s^{n-j}t^j\,\sum_{|\rmA|=n-j}\left(X^{\vee j}\right)_{\rmA',\rmA'}\\
&=\sum_{0\le j\le n} s^{n-j}t^j\,\tr X^{\vee j}
\end{align*}
as required.
\end{proof}
\end{subsection}

\begin{subsection}{Permutation groups}
Let $X$ be an $n\times n$ permutation matrix. We can express $\per(I+tX)$ in terms of the cycle decomposition
of the associated permutation.

\begin{proposition} \label{prop:cycles} For a permutation matrix $X$,
$$\per(I+tX)=\prod_{0\le \ell\le n}(1+t^\ell)^{n_X(\ell)}$$
where $n_X(\ell)$ is the number of cycles of length $\ell$ in the cycle decomposition of the 
corresponding permutation.
\end{proposition}
\begin{proof}Decomposing the permutation associated to $X$ yields a decomposition into
invariant subspaces of the underlying vector space $\V$. So $\per(I+tX)$ will be the product
of $\per(I+tX_c)$ as $c$ runs through the corresponding cycles with $X_c$ the restriction of $X$ to the invariant subspace for each $c$. 
So we have to check that if $X$ acts on $\V^{\ell}$ as a cycle of length $\ell$, then $\per(I+tX)=1+t^\ell$. For this,
apply Theorem \ref{thm:ztrace}. Apart from level zero, there is only one set fixed by any $X^{\vee j}$, namely
when $j=\ell$. So the trace of $X^{\vee j}$ is zero unless $j=\ell$ and then it is one. The result follows. 
\flushright\qedhere
\end{proof}
\begin{subsubsection}{Cycle index. Orbits on $\ell$-sets.}
Now consider a group, $G$, of permutation matrices. We have the cycle index
$$Z_G(z_1, z_2,\ldots,z_n)=\frac{1}{|G|}\,\sum_{X\in G} z_1^{n_X(1)}z_2^{n_X(2)}\cdots z_n^{n_X(n)}$$
each $z_\ell$ corresponding to $\ell$-cycles in the cycle decomposition associated to the $X$'s.
From Proposition \ref{prop:cycles},  we have an expression in terms of permanents. Combining with the trace
formula, we get
\begin{theorem} Let $G$ be a permutation group of matrices. Then we have
\begin{align*}
\frac{1}{|G|}\,\sum_{X\in G} \per(I+tX) &= Z_G(1+t,1+t^2,\ldots,1+t^n)\\ &=\sum_{\ell} t^{\ell}\, \#  (\mathrm{\,orbits\ on\ }\ell\text{\rm  -sets\,})\ .
\end{align*}
\end{theorem}
\begin{remark} This result refers to three essential theorems in group theory acting on sets. Equality of the first and last expressions
is the `permanent' analog of Molien's Theorem, which is the case for a group acting on the symmetric tensor algebra. That the cycle index
counts orbits on subsets is an instance of Polya Counting, with two colors. The last expression follows by the
Cauchy-Burnside Lemma applied to the groups $G^{\vee\ell}=\{X^{\vee\ell}\}_{X\in G}$.
\end{remark}
\end{subsubsection}

\begin{subsubsection}{Centralizer algebra and Johnson scheme}
Given a group, $G$, of permutation matrices, an important question is to determine 
the set (among all matrices) of matrices commuting with all of the matrices in $G$. This is the 
\textsl{centralizer algebra} of the group. For the symmetric group, the only matrices are $I$ and $J$.
For the action of the symmetric group on $\ell$-sets, a basis for the centralizer algebra is given by the
incidence matrices for the Johnson distance. These are the same as the adjacency matrices for the Johnson
(association) scheme. Recall that the Johnson distance between two $\ell$-sets $\I$ and $\J$ is
$$\text{dist}_{\rm JS}(\I,\J)={\textstyle\frac12}\, |\I\DD\J|=|\I\setminus\J|=|\J\setminus\I|$$
The corresponding matrices $\JS{k}{n\ell}$ are defined by
$$(\JS{k}{n\ell})_{\I\J}=\begin{cases} 1, & \text{if dist}_{\rm JS}(\I,\J)=k \\ 0& \text{otherwise}\end{cases}$$
As it is known, \cite[p.\,36]{CAM}, that a basis for the centralizer algebra is given by the orbits of the group $G^2$,
acting on pairs, the Johnson basis is a basis for the centralizer algebra.
Since the Johnson distance is symmetric, it suffices to look at $G^{\vee 2}$. \bigskip

Now we come to the question that is one starting point for this work. If $I$ and $J$ are the only matrices commuting with
all elements (as matrices) of the symmetric group, then since the map $G\to G^{\vee\ell}$ is a homomorphism, we know that
$I^{\vee\ell}$ and $J^{\vee\ell}$ are in the centralizer algebra of $G^{\vee\ell}$. The question is: how to obtain the rest?
The, perhaps surprising, answer is that in fact one can obtain the complete Johnson basis from $I$ and $J$ alone.
This will be one of the main results, Theorem \ref{thm:xnl}.
\end{subsubsection} 

\begin{subsubsection}{Permanent of $sI+tJ$}
First, let us  consider $sI+tJ$.
\begin{proposition}
We have the formula
\begin{equation}\label{eq:perst}
\per(sI+tJ)=n!\,\sum_{0\le \ell\le n} \frac{s^\ell t^{n-\ell}}{\ell!}
\end{equation}
\end{proposition}
\begin{proof}
For $X=J$, we see directly, since all entries equal one in all submatrices, that
$$(J^{\vee\ell})_{\I\J}=\ell!$$
for all $\I$ and $\J$. Taking traces
$$\tr J^{\vee\ell}=\binom{n}{\ell}\,\ell!$$
and by the trace formula, Theorem \ref{thm:ztrace},
$$\per(sI+tJ)=\sum_\ell \binom{n}{\ell}\,\ell!\,s^{n-\ell}t^\ell=\sum_\ell \frac{n!}{(n-\ell)!}\,s^{n-\ell}t^\ell\ .$$
Reversing the order of summation yields the result stated.
\end{proof}
\begin{corollary}
For varying $n$, we will explicitly denote $p_n(s,t)=\per (sI_n+tJ_n)$. Then, with \hfill\break $p_0(s,t)=1$,
$$\sum_{n=0}^\infty \frac{z^n}{n!}p_n(s,t)=\frac{e^{sz}}{1-tz}\ .$$
\end{corollary}

The Corollary exhibits the operational formula
$$p_n(s,t)=\frac{1}{1-tD_s}s^n$$
where $D_s=d/ds$. By inspection, this agrees with \eqref{eq:perst} as well.

Observe that equation \eqref{eq:perst} can be rewritten as
$$\per(sI+tJ)=\int_0^\infty (s+ty)^n\,e^{-y}\,dy$$
that is, these are ``moment polynomials" for the exponential distribution with an additional scale parameter.\bigskip

We proceed to examine these moment polynomials in detail.
\end{subsubsection} 
\end{subsection} 

\section{Exponential polynomials}
For the exponential distribution, with density $e^{-y}$ on $(0,\infty)$, the \textsl{moment polynomials} are defined as
$$h_n(x)=\int_0^\infty (x+y)^n\,e^{-y}\,dy$$
The exponential embeds naturally into the family of weights of the form $x^m\,e^{-x}$ on $(0,\infty)$ as for generalized Laguerre polynomials.
We define correspondingly
\begin{equation}\label{eq:hnm}
h_{n,m}(x,t)=\int_0^\infty (x+ty)^n\, (ty)^m\,e^{-y}\,dy
\end{equation}
for nonnegative integers $n,m$, introducing a factor of $y^m$ and a scale factor $t$. 
We refer to these as \textsl{exponential moment polynomials}.

\begin{proposition} \label{prop:eply}
Observe the following properties of the exponential moment polynomials. \bigskip

1. The generating function
$$ \frac{1}{t^m m!}\,\sum_{n=0}^\infty \frac{z^n}{n!}h_{n,m}(x,t)=\frac{e^{zx}}{(1-tz)^{1+m}}$$
for  $|tz|<1$. \bigskip

2. The operational formula
$$\frac{1}{t^m m!}\,h_{n,m}(x,t)=(I-tD)^{-(m+1)} x^n$$
where $I$ is the identity operator and $D=d/dx$.  \bigskip

3. The explicit form
$$h_{n,m}(x,t)=\sum_{j=0}^n\binom{n}{j}(m+j)!\,x^{n-j}t^{m+j}$$
\end{proposition}
\begin{proof} For the first formula, multiply the integral by $z^n/n!$ and sum to get
$$\int_0^\infty y^m\,e^{zx+zty-y}\,dy=e^{zx}\,\int_0^\infty y^m\,e^{-y(1-tz)}\,dy$$
which yields the stated result.  \bigskip

For the second, write
$$t^mm!\,(I-tD)^{-(m+1)}x^n=t^m\int_0^\infty y^m\,e^{-(I-tD)y}x^n\,dy=\int_0^\infty (ty)^m\,e^{-y}(x+ty)^n\,dy$$
using the shift formula $e^{aD}f(x)=f(x+a)$. \bigskip

For the third, expand $(x+ty)^n$ by the binomial theorem and integrate.
\end{proof}

A variation we will encounter in the following is
\begin{align} 
h_{n-m,m}(x,t)&=\sum_{j=0}^{n-m}\binom{n-m}{j}(m+j)!\,x^{n-m-j}t^{m+j}  \label{eq:exppoly0}\\
&=\sum_{j=m}^n\binom{n-m}{j-m}j!\,x^{n-j}t^j \label{eq:exppoly}\\
&=\sum_{j=m}^n\binom{n-m}{n-j}j!\,x^{n-j}t^j   \label{eq:exppoly1}\\
&=\sum_{j=0}^{n-m}\binom{n-m}{j}(n-j)!\,x^{j}t^{n-j}  \label{eq:exppoly2}
\end{align}
replacing the index $j\gets j-m$ for \eqref{eq:exppoly} and reversing the order of summation for the last line.  And for future reference, the integral formula,
\begin{equation}\label{eq:hnmm}
h_{n-m,m}(x,t)=\int_0^\infty (x+ty)^{n-m}\, (ty)^m\,e^{-y}\,dy
\end{equation}

\begin{subsection}{Hypergeometric form}
Generalized hypergeometric functions provide expressions for the exponential moment polynomials that are often convenient.
In the present context we will use ${}_2 F_0$ functions, defined by
$${}_2 F_0\left(\genfrac{}{}{0pt}{}{ a ,b}{\text{---}}\biggm| x \right)=\sum_{j=0}^\infty \frac{(a)_j\,(b)_j}{j!}\,x^j$$
where $(a)_j=\Gamma(a+j)/\Gamma(a)$ is the usual Pochhammer symbol. In particular, if $a$, e.g., is a negative integer, the series
reduces to a polynomial. Rearranging factors in the expressions for $h_{n,m}$, via \#3 in the above Proposition, and $h_{n-m,m}$, 
eq.\,\eqref{eq:exppoly0}, we can formulate these as ${}_2F_0$ hypergeometric functions.

\begin{proposition}\label{prop:hypgeo} We have the following expressions for exponential moment polynomials.
\begin{align*}
h_{n,m}(x,t)&=x^nt^m\,m!\,{}_2 F_0\left(\genfrac{}{}{0pt}{}{ -n ,1+m}{\text{---}}\biggm| -\frac{t}{x} \right)
\phantom{\biggm |}\\
h_{n-m,m}(x,t)&=x^{n-m}t^m\,m!\,{}_2 F_0\left(\genfrac{}{}{0pt}{}{ m-n ,1+m}{\text{---}}\biggm| -\frac{t}{x} \right)\ .
\end{align*}
\end{proposition}
\end{subsection}

\section{Zeon powers of $sI+tJ$}\label{sec:ZP}

We want to calculate $(sI+tJ)^{\vee\ell}$, i.e., the $\binom{n}{\ell}\times\binom{n}{\ell}$ matrix with rows and columns labelled by
$\ell$-subsets $\I,\J\subset\{1,\ldots,n\}$ with the $\I\J$ entry equal to the permanent of the corresponding submatrix of 
$sI+tJ$. This is equivalent to the induced action of the original matrix $sI+tJ$ on the $\ell^{\rm th}$ zeon space $\V^{\vee\ell}$. \bigskip

\begin{theorem} \label{thm:xnl} The $\ell^{\rm th}$ zeon power of $sI+tJ$ is given by
$$(sI+tJ)^{\vee\ell}=\sum_k\sum_{j=k}^\ell\binom{\ell-k}{\ell-j}\,j!\,s^{\ell-j}t^j \,\JS{k}{n\ell}=\sum_k \,h_{\ell-k,k}(s,t)\,\JS{k}{n\ell}$$
where the $h$'s are exponential moment polynomials.
\end{theorem}
\begin{proof}
Choose $\I$ and $\J$ with $|\I|=|\J|=\ell$.
By Theorem \ref{thm:zpowers}, we have, using the fact that all of the entries of $J^{\vee j}$ are equal to $j!$,
\begin{align*}
\left((sI+tJ)^{\vee\ell}\right)_{\I\J}&=
\sum_{0\le j\le\ell} s^{\ell-j}t^j\,\sum_{\substack{A\subset \I\cap\J\\|\rmA|=\ell-j}}\left(J^{\vee j}\right)_{\I \setminus \rmA,\J \setminus \rmA}\\
&=\sum_{0\le j\le\ell} s^{\ell-j}t^j\,\sum_{\substack{A\subset \I\cap\J\\|\rmA|=\ell-j}}\,j!
\end{align*}
Now, if $\text{dist}_{\rm JS}(\I,\J)=k$, then $|\I\cap\J|=\ell-k$ and there are $\displaystyle \binom{\ell-k}{\ell-j}$ subsets $\rmA$ of
$\I\cap\J$ satisfying the conditions of the sum. Hence the result.
\end{proof}

Note that the specialization $\ell=n$, $k=0$, recovers equation \eqref{eq:perst}.  \bigskip

We can write the above expansion using the hypergeometric form of the exponential moment polynomials, Proposition \ref{prop:hypgeo},
$$(sI+tJ)^{\vee\ell}=
\sum_k s^{\ell-k}t^k\,k!\,{}_2 F_0\left(\genfrac{}{}{0pt}{}{ k-\ell ,1+k}{\text{---}}\biggm| -\frac{t}{s} \right)\,\JS{k}{n\ell}\ .$$

\subsection{Spectrum of the Johnson matrices}
Recall, e.g., \cite[p.\,220]{BI},  that the spectrum of the Johnson matrices for given $n$ and $\ell$ are the numbers 
\begin{equation}\label{eq:jspec}
\Lambda_k^{n\ell}(\alpha)=\sum_i \binom{\ell-\alpha}{i}\binom{n-\ell-\alpha+i}{i}\binom{\ell-i}{k-i}\,(-1)^{k-i}
\end{equation}
where the eigenvalue for given $\alpha$ has multiplicity $\displaystyle \binom{n}{\alpha}-\binom{n}{\alpha-1}$. \bigskip

For $\ell$-sets, the Johnson distance takes values from $0$ to $\min(\ell,n-\ell)$, with $\alpha$ taking values from 
that same range.

\subsection{The spectrum of $(sI+tJ)^{\vee\ell}$}
Recall that as the Johnson matrices are symmetric and generate a commutative algebra, they are simultaneously diagonalizable
by an orthogonal transformation of the underlying vector space. 
Diagonalizing the equation in Theorem \ref{thm:xnl}, we see that the spectrum of $(sI+tJ)^{\vee\ell}$ is given by
$$\sum_k \,h_{\ell-k,k}(s,t)\,\Lambda_k^{n\ell}(\alpha)\ .$$

\begin{proposition} \label{prop:spec} The spectrum of $(sI+tJ)^{\vee\ell}$ is given by
$$\frac{s^\alpha}{t^{n-\ell-\alpha}({\scriptstyle n-\ell-\alpha})!}\,h_{\ell-\alpha,n-\ell-\alpha}(s,t)
=\sum_i s^{\ell-i}t^i\,\binom{\ell-\alpha}{i}\binom{n-\ell-\alpha+i}{i}\,i!\,$$
for $0\le \alpha \le \min(\ell,n-\ell)$, with respective multiplicities $\displaystyle  \binom{n}{\alpha}-\binom{n}{\alpha-1}$.
\end{proposition}
\begin{proof} In the sum over $i$ in equation \eqref{eq:jspec}, only the last two factors involve $k$. We have
\begin{align*}
\sum_kh_{\ell-k,k}(s,t)(-1)^{k-i}\binom{\ell-i}{k-i}&=\sum_k\int_0^\infty(s+ty)^{\ell-k}(ty)^k(-1)^{k-i}\binom{\ell-i}{k-i}\,e^{-y}\,dy\\
\text{setting $k=i+m$}&\\
&=\sum_m\int_0^\infty(s+ty)^{\ell-i-m}(ty)^{i+m}(-1)^{m}\binom{\ell-i}{m}\,e^{-y}\,dy\\
&=\int_0^\infty(s+ty-ty)^{\ell-i}(ty)^{i}\,e^{-y}\,dy\\
&=s^{\ell-i}t^i i!
\end{align*}
using the binomial theorem to sum out $m$. Filling in the additional factors yields
$$\sum_k \,h_{\ell-k,k}(s,t)\,\Lambda_k^{n\ell}(\alpha)=\sum_is^{\ell-i}t^i\, i!\,\binom{\ell-\alpha}{i}\binom{n-\ell-\alpha+i}{i}\ .$$
Taking out a denominator factor of $(n-\ell-\alpha)!$ and multiplying by $s^{-\alpha}t^{n-\ell-\alpha}$ gives
$$\sum_i s^{\ell-\alpha-i}t^{n-\ell-\alpha+i}\binom{\ell-\alpha}{i}(n-\ell-\alpha+i)!$$
which is precisely $h_{\ell-\alpha,n-\ell-\alpha}$ as in the third statement  of Proposition \ref{prop:eply}.
\end{proof}

As in Proposition \ref{prop:hypgeo}, we can express the eigenvalues as follows.
\begin{corollary} \label{cor:spec}
The spectrum of $(sI+tJ)^{\vee\ell}$ consists of the eigenvalues
$$s^\ell\,{}_2 F_0\left(\genfrac{}{}{0pt}{}{\alpha-\ell ,1+n-\ell-\alpha}{\text{---}}\biggm| -\frac{t}{s} \right)$$
for $0\le \alpha\le \min(\ell,n-\ell)$, with corresponding multiplicities as indicated above.
\end{corollary}
\begin{subsection}{Row-sums and trace identity}
For the row-sums, we know that the all-ones vector is a common eigenvector of the Johnson basis corresponding to $\alpha=0$. These
are the valencies $\Lambda_k(0)$. For the Johnson scheme, we have
$$\Lambda_k^{n\ell}(0)=\binom{\ell}{k}\binom{n-\ell}{k}$$
e.g., see \cite[p.\,219]{BI}, which can be checked directly from the formula for $\Lambda_k^{n\ell}(\alpha)$, equation \eqref{eq:jspec}, with $\alpha$ set to zero. 
Setting $\alpha=0$ in Proposition \ref{prop:spec} gives 
\begin{equation}\label{eq:rowsums}
\frac{1}{t^{n-\ell}\,(n-\ell)!}\,h_{\ell,n-\ell}(s,t)=\sum_i\binom{\ell}{i}\binom{n-\ell+i}{i}\,i!\, s^{\ell-i}t^i
\end{equation}
for the row-sums of $(sI+tJ)^{\vee\ell}$ .
\begin{subsubsection}{Trace Identity}
Terms on the diagonal are the coefficient of $\JS{0}{n\ell}$, which is the identity matrix. So the trace is 
$$\tr (sI+tJ)^{\vee\ell}=\binom{n}{\ell}\,h_{\ell,0}(s,t) =\binom{n}{\ell}\sum_k\binom{\ell}{k}k!\,s^{\ell-k}t^k$$
Cancelling factorials and reversing the order of summation on $k$ yields the formula:
\begin{proposition}\label{prop:trace}
$$\tr (sI+tJ)^{\vee\ell}=\frac{n!}{(n-\ell)!}\,\sum_{0\le k\le \ell}\frac{s^kt^{\ell-k}}{k!}$$
\end{proposition}
Now, Proposition \ref{prop:spec} gives the trace
$$
\tr (sI+tJ)^{\vee\ell}=\sum_{0\le\alpha\le\min(\ell,n-\ell)}\left[\binom{n}{\alpha}-\binom{n}{\alpha-1}\right]\,
\sum_i s^{\ell-i}t^i\,\binom{\ell-\alpha}{i}\binom{n-\ell-\alpha+i}{i}\,i!$$
Equating the above expressions for the trace yields the identity
\begin{align*}
\sum_{0\le\alpha\le\min(\ell,n-\ell)}\left[\binom{n}{\alpha}-\binom{n}{\alpha-1}\right]\,
&\sum_i s^{\ell-i}t^i\,\binom{\ell-\alpha}{i}\binom{n-\ell-\alpha+i}{i}\,i!\\ 
&=\frac{n!}{(n-\ell)!}\,\sum_{0\le j\le \ell}\frac{s^j t^{\ell-j}}{j!}
\end{align*}

\end{subsubsection}
\end{subsection}
\hfill \bigskip

\begin{example} For $n=4$, $\ell=2$ we have
 $$\left[ \begin {array}{cccccc} 
s^2+2st+2t^2&st+2t^2&st+2t^2&st+2t^2&st+2t^2&2t^2\\
st+2t^2&s^2+2st+2t^2&st+2t^2&st+2t^2&2t^2&st+2t^2\\
st+2t^2&st+2t^2&s^2+2st+2t^2&2t^2&st+2t^2&st+2t^2\\
st+2t^2&st+2t^2&2t^2&s^2+2st+2t^2&st+2t^2&st+2t^2\\ 
st+2t^2&2t^2&st+2t^2&st+2t^2&s^2+2st+2t^2&st+2t^2\\
2t^2&st+2t^2&st+2t^2&st+2t^2&st+2t^2&s^2+2st+2t^2
\end {array} \right] $$
One can check that the entries are in agreement with Theorem \ref{thm:xnl}. The trace is
$6s^2+12st+12t^2$. The spectrum is
\begin{align*}
&\text{eigenvalue } s^2+6st+12t^2, &\text{ with multiplicity } 1&\\
&\text{eigenvalue } s^2+2st, &\text{ with multiplicity } 3&\\
&\text{eigenvalue } s^2, &\text{ with multiplicity } 2&
\end{align*}
and the trace can be verified from these as well.
\end{example}

\begin{remark}
What is interesting is that these matrices have polynomial entries with all eigenvalues polynomials as well and furthermore,
the exact same set of polynomials produces the eigenvalues as well as the entries. Specializing $s$ and $t$ to integers, a similar
statement holds. All of these matrices will have integer entries with integer eigenvalues all of which belong to closely related families of numbers.
We will examine interesting cases of this phenomenon later on in this paper.
\end{remark}

\section{Permanents from $sI+tJ$}\label{S:gg}
Here we present a proof via recursion of the subpermanents of $sI+tJ$, thereby recovering 
Theorem \ref{thm:xnl} from a different perspective. \bigskip

\begin{remark} For the remainder of this paper, we will work with an $n\times n$ matrix
corresponding to an $\ell\times\ell$ submatrix of the above discussion. Here we have
blown up the submatrix to full size as the object of consideration.
\end{remark}

Let $M_{n,\ell}$ denote the $n\times n$ matrix with $n-\ell$ entries equal to $s+t$ on the main diagonal, 
and $t$'s elsewhere. 
Note that $M_{n,0}=sI+tJ$ and $M_{n,n}=tJ$, where $I$ and $J$ are $n\times n$. 
Define 
\begin{equation}\label{eq:perM}
P_{n,\ell}=\per (M_{n,\ell})
\end{equation}
 to be the permanent of $M_{n,\ell}$.
\smallskip

For $\ell=0$, define
$P_{0,0}=1$, and, recalling equation \eqref{eq:perst},
\begin{equation}\label{eq:init}
P_{n,0}=\per (sI+tJ)=\sum_{j=0}^n\frac{n!}{j!}s^jt^{n-j}=\sum_{j=0}^n\frac{n!}{(n-j)!}s^{n-j}t^j\ .
\end{equation}
We have also $P_{n,n}=\per (tJ)=n!\,t^n$ for $J$ of order $n\times n$. These agree at
$P_{0,0}=1$.

\begin{theorem}\label{T:Pnlrecurrence}
For $n\ge1$, $1\le \ell\le n$, we have the recurrence
\begin{equation}\label{E:Pnlrecurrence}
P_{n,\ell}=P_{n,\ell-1}-sP_{n-1,\ell-1}.
\end{equation}
\end{theorem}
\begin{proof}
We have $0\le \ell\le n$ so $n-(\ell-1)=n-\ell+1\ge1$, i.e.,
the matrix $M_{n,\ell-1}$ contains at least $1$ entry on its main diagonal equal to $s+t$. 
Write the block form
$$M_{n, \ell-1}=\begin{bmatrix}s+t&A\\A^T&M_{n-1,\ell-1} \end{bmatrix}$$
with $A=[t,t,\ldots t]$ the $1\times (n-1)$ row vector of all $t$'s, and $A^T$ its transpose.
Now compute the permanent of $M_{n,\ell-1}$ expanding along the first row. We get
\begin{equation}\label{E:PPP}
P_{n, \ell-1}=\per(M_{n, \ell-1})=(s+t)\,\per(M_{n-1,\ell-1})+F(A,A^T,M_{n-1,\ell-1})
\end{equation}
where 
$F(A,A^T,M_{n-1,\ell-1})$ is the contribution to $P_{n,\ell-1}$ involving $A$. Now 
\begin{align*}
t\,\per(M_{n-1,\ell-1})+F(A, & A^T, M_{n-1,\ell-1})\\
&=\per\big(\begin{bmatrix}t&A\\A^T&M_{n-1,\ell-1}\end{bmatrix}\big)\\
&=\per\big(\begin{bmatrix}A^T&M_{n-1,\ell-1}\\t&A\end{bmatrix}\big)\\
&=\per\big(\begin{bmatrix}M_{n-1,\ell-1}&A^T\\A&t\end{bmatrix}\big)\\
&=P_{n,\ell}\ .
\end{align*}

Thus from equation \eqref{E:PPP}:
\begin{align*}
P_{n, \ell-1}&=s\,\per(M_{n-1,\ell-1})
+t\,\per(M_{n-1,\ell-1})
+F(A,A^T,M_{n-1,\ell-1})\\
&=s\,P_{n-1,\ell-1}+P_{n,\ell}
\end{align*}
and so the result. \flushright\qedhere
\end{proof}

We arrange the polynomials $P_{n,\ell}$ in a triangle, with the columns labelled by $\ell\ge0$ and rows by $n\ge0$, starting with $P_{0,0}=1$ at the top
vertex:
$$\begin{array}{ccccc}
P_{0,0}&&&&\\
P_{1,0}&P_{1,1}&&&\\
P_{2,0}&P_{2,1}&P_{2,2}&&\\
\vdots&&\ddots&\ddots& \\
P_{n-1,0}&\multicolumn{2}{c}{\hdots}&P_{n-1,n-2}&P_{n-1,n-1}\\
P_{n,0}&P_{n,1}&\hdots&P_{n,n-1}&P_{n,n}
\end{array}
$$
The recurrence says that to get the $n,\ell$ entry, you combine elements in column $\ell-1$ in rows $n$ and $n-1$, forming an {\large\tt L}-shape. 
Thus, given the first column $\{P_{n,0}\}_{n\ge0}$, the table can be generated in full. \bigskip

Now we check that these are indeed our exponential moment polynomials.
Additionally we derive an expression for $P_{n,\ell}$ in terms of the initial sequence $P_{n,0}$. For clarity, we will explicitly denote the dependence
of $P_{n,\ell}$ on $(s,t)$.

\begin{theorem}\label{thm:Pst=}
For $\ell\ge 0$ we have \bigskip

1.  The permanent of the $n\times n$ matrix with $n-\ell$ entries on the diagonal equal to $s+t$ and all other entries equal to $t$ is
\begin{equation}\label{E:Pst1}
P_{n,\ell}(s,t)=h_{n-\ell,\ell}(s,t)=\sum_{j=\ell}^n\binom{n-\ell}{n-j}j!\,s^{n-j}t^j\ .
\end{equation}

\hfill \bigskip

2. \begin{equation}\label{E:Pst2}
P_{n,\ell}(s,t)=\sum_{j=0}^{\ell}\binom{\ell}{j}{(-1)^j}s^jP_{n-j,0}(s,t)\ .
\end{equation} \bigskip

3. And the complementary sum: $$s^n=\sum_{j=0}^n\binom{n}{\ell} (-1)^\ell\,P_{n,\ell}(s,t)\ .$$
\end{theorem}

\begin{proof}
The initial sequence $P_{n,0}=h_{n,0}$ as noted in equation \eqref{eq:init}. 
We check that $h_{n-\ell,\ell}$ satisfies recurrence \eqref{E:Pnlrecurrence}. Starting from the integral representation for $h_{n-\ell+1,\ell-1}$, 
equation \eqref{eq:hnm}, we have
\begin{align*}
h_{n-\ell+1,\ell-1}&=\int_0^\infty (s+ty)^{n-\ell+1}(ty)^{\ell-1}\,e^{-y}\,dy \\
&=\int_0^\infty (s+ty)(s+ty)^{n-\ell}(ty)^{\ell-1}\,e^{-y}\,dy \\
&=s\,h_{n-\ell,\ell-1}+h_{n-\ell,\ell}\\
\end{align*}
as required, where we now identify $h_{n-\ell+1,\ell-1}=P_{n,\ell-1}$, 
$h_{n-\ell,\ell-1}=P_{n-1,\ell-1}$, and $h_{n-\ell,\ell}=P_{n,\ell}$. And equation \eqref{eq:exppoly1}
gives an explicit form for $P_{n,\ell}$.\bigskip

For \#2, starting with the integral representation for $P_{n,0}=h_{n,0}$, we get
\begin{align*}
\sum_{j=0}^{\ell}\binom{\ell}{j}{(-1)^j}s^j\,\int_0^\infty (s+ty)^{n-j}\,e^{-y}\,dy
&=\sum_{j=0}^{\ell}\binom{\ell}{j}{(-1)^j}s^j\,\int_0^\infty (s+ty)^{n-\ell}(s+ty)^{\ell-j}\,e^{-y}\,dy\\
&=\int_0^\infty (s+ty)^{n-\ell}(s+ty-s)^\ell\,e^{-y}\,dy\\
&=h_{n-\ell,\ell}
\end{align*}
as required. The proof for \#3 is similar, using equation \eqref{eq:hnmm},
$$P_{n,\ell}=h_{n-\ell,\ell}=\int_0^\infty (s+ty)^{n-\ell}(ty)^\ell\,e^{-y}\,dy$$
and the binomial theorem for the sum.
\end{proof}

\begin{subsection}{$(sI+tJ)^{\vee\ell}$ revisited}
Now we have an alternative proof of Theorem \ref{thm:xnl}.
\begin{lemma} \label{L:more}
Let $\I$ and $\J$ be $\ell$-subsets of $[n]$ with $\text{dist}_{\rm JS}(\I,\J)=k$.
Then
\begin{equation*}
\per(sI+tJ)_{{\rm IJ}}=P_{\ell, k}(s,t).
\end{equation*}
\end{lemma}
\begin{proof}
Now $|{\rm I}\cap{\rm J}|=\ell-k$ so the submatrix 
$(sI+tJ)_{{\rm IJ}}$
is permutationally equivalent to  the $\ell\times\ell$ matrix with
$\ell-k$ entries $s+t$ on its main diagonal and $t$'s elsewhere, i.e., to the 
matrix $M_{\ell, k}$. 
Hence, by definition of $P_{\ell, k}(s,t)$, equation \eqref{eq:perM}, we have the result.
\end{proof}

Thus, the expansion in the Johnson basis:

\begin{equation}\label{E:newproof}
(sI+tJ)^{\vee\ell}=\sum_{k}h_{\ell-k,k}(s,t)\,{\rm JS}^{n\ell}_k
\end{equation}

\begin{proof}
Let {\rm I} and {\rm J} be $\ell$-subsets of $[n]$ with Johnson-distance $k$. 
By definition, the IJ entry of the LHS of equation \eqref{E:newproof}
equals the permanent of the submatrix from rows $\I$ and columns $\J$,
 ${\rm per}(sI+tJ)_{\I\J}=P_{\ell, k}(s,t)=h_{\ell-k,k}(s,t)$, 
by the above Lemma and Theorem \ref{thm:Pst=}, \#1. 
Now on the RHS of equation~(\ref{E:newproof}), if $\text{dist}_{\rm JS}(\I,\J)=k$,
the only nonzero contribution comes from the ${\rm JS}^{n\ell}_k$ term.
This yields $h_{\ell-k,k}(s,t)\times 1=h_{\ell-k,k}(s,t)$ as required.
\end{proof}
\end{subsection} \bigskip

\begin{subsection}{Elementary subgraphs and permanents}
There is an approach to permanents of $sI+tJ$ via elementary subgraphs, based on that of Biggs~\cite{BIGGS} 
for determinants. \medskip

An {\em elementary subgraph\/} (see \cite[p.\,44]{BIGGS}) of a graph 
$G$ is a spanning subgraph of $G$ all of whose components are
$0$, $1$, or $2$-regular, {i.e.}, 
all of whose components are isolated vertices, isolated edges, or cycles of length $j\ge 3$. \bigskip

Let $K_n^{(\ell)}$ be a copy of the complete graph $K_n$ with vertex set $[n]$ in which 
the first 
$n-\ell$ vertices $[n-\ell]=\{1,2,\ldots,n-\ell\}$ are  
 {\em distinguished\/}.  We may now consider the matrix $M_{n,\ell}$ as the {\em weighted\/} adjacency matrix of 
$K_n^{(\ell)}$ in which 
the weights of the distinguished vertices are $s+t$, with
all undistinguished vertices and all edges assigned a weight of $t$. \bigskip

Let $E$ be an elementary subgraph of  $K_n^{(\ell)}$. Then we describe $E$ as having $d(E)$ 
distinguished isolated  vertices and $c(E)$ cycles. 
The weight of $E$, $\text{wt}(E)$, is defined as:
\begin{equation}\label{E:wtE}
\text{wt}(E)=(s+t)^{d(E)}t^{n-d(E)},  
\end{equation}
a  homogeneous polynomial of degree $n$. \bigskip

This leads to an interpretation/derivation of $P_{n,\ell}(s,t)$ as the permanent $\per(M_{n,\ell})$.

\begin{theorem} \label{T:Esubgraph}
We have the expansion in elementary subgraphs
$$
P_{n,\ell}(s,t)=\sum_{E}2^{c(E)}\,{\rm wt}(E)\ .
$$
\end{theorem}
\begin{proof}
Assign weights to the components of $E$ as follows: \medskip

each distinguished isolated  vertex will have weight $s+t$; \par
each undistinguished isolated  vertex will have weight $t$;\par
each isolated  edge will have weight $t^2$; \par
and each $j$-cycle, $j\ge 3$, will have weight $t^j$. \medskip

To obtain wt$(E)$ in agreement with equation~(\ref{E:wtE}) we 
form the product of these weights  
over all components in $E$. 
The proof then follows along the lines of Proposition~7.2 of \cite[p.44]{BIGGS}, 
slightly modified to incorporate isolated vertices and with determinant, `${\det}$',
replaced by permanent, `${\per}$', ignoring the minus signs. 
Effectively, each term in the permanent expansion thus
corresponds to a weighted elementary subgraph $E$ of the weighted $K_n^{(\ell)}$. 
\end{proof}
\bigskip

An example with $n=3$ is on the last page of the Appendix. 
\end{subsection}

\begin{subsection}{Associated polynomials and some asymptotics}
Thinking of $s$ and $t$ as parameters, we define the \textsl{associated polynomials}
$$Q_n(x)=\sum_{\ell=0}^n\binom{n}{\ell}x^\ell P_{n,\ell}\ .$$
As in the proof of \#3 above, using the integral formula \eqref{eq:hnmm}, we have
\begin{align}\label{eq:Qn}
Q_n(x)&=\int_0^\infty(s+ty+xty)^n\,\,e^{-y}\,dy \nonumber\\
&=\sum_j \binom{n}{j}s^j(1+x)^{n-j}t^{n-j}(n-j)!\nonumber\\
&=n!\sum_j\frac{s^j(1+x)^{n-j}t^{n-j}}{j!}\ .
\end{align}
Comparing with equation \eqref{eq:init}, we have
\begin{proposition}
$$Q_n(x)=\sum_{\ell=0}^n\binom{n}{\ell}x^\ell P_{n,\ell}(s,t)=P_{n,0}(s,t+xt)\ .$$
\end{proposition}
And we have 
\begin{proposition}
As $n \to \infty$, for $x\ne -1$,
$$Q_n(x)\sim t^n(1+x)^n\,n!\,e^{s/(t+tx)}$$
with the special cases
\begin{gather*}
Q_n(-1)=s^n\\ \mathstrut\\
Q_n(0)=P_{n,0}\sim t^nn!\,e^{s/t}\\\mathstrut\\
Q_n(1)=\sum_\ell\binom{n}{\ell}P_{n,\ell}\sim (2t)^nn!\,e^{s/(2t)}
\end{gather*}
\end{proposition}
\begin{proof} From equation \eqref{eq:Qn}
\begin{align*}
Q_n(x)&=n!\sum_j\frac{s^j(1+x)^{n-j}t^{n-j}}{j!}\\
&=t^n(1+x)^nn!\,\sum_{j=0}^n \frac{1}{j!}\left(\frac{s/t}{1+x}\right)^j
\end{align*}
from which the result follows.
\end{proof}
\end{subsection}

\section{Generalized derangement numbers}

The formula \eqref{eq:perst} is suggestive of the derangement numbers (see, e.g., \cite[p.\,180]{COM}),
$$d_n=n!\sum_{j=0}^n\frac{(-1)^j}{j!}\ .$$
This leads  to   \bigskip

\begin{definition}     
A family of numbers, depending on $n$ and $\ell$,
 arising as the values of $P_{n,\ell}(s,t)$ when $s$ and $t$ are assigned fixed integer values, are called
\textsl{generalized derangement numbers}.
\end{definition}
\hfill\bigskip

We have seen that the assignment $s=-1,\,t=1$ produces the usual derangement numbers when $\ell=0$.
In this section, we will examine in detail the cases $s=-1\,,t=1$, generalized \textsl{derangements}, and $s=t=1$,
generalized \textsl{arrangements}. \bigskip

\begin{remark}
Topics related to this material are discussed in Riordan, \cite{R}. The article \cite{ST} is of related interest as well.
\end{remark}

\subsection{Generalized derangements of $[n]$}\label{SS:derangements}
To start, define $$D_{n,\ell}=P_{n,\ell}(-1,1)\ .$$ 
Equation \eqref{E:Pst1} and Proposition \ref{prop:hypgeo} give:
\begin{equation}\label{E:Dnell}
D_{n,\ell}=\sum_{j=\ell}^{n}(-1)^{n-j}\binom{n-\ell}{n-j}{j!}=
(-1)^{n-\ell}\,\ell!\,{}_2 F_0\left(\genfrac{}{}{0pt}{}{ \ell-n ,1+\ell}{\text{---}}\biggm| 1 \right)\ .
\end{equation}
Equation \eqref{eq:init} reads
$$\per(J-I)=D_{n,0}=d_n$$ 
the number derangements of $[n]$. So we have
a combinatorial interpretation of $D_{n,0}$.

\begin{subsubsection}{Combinatorial interpretation of $D_{n,\ell}$}
We now give a combinatorial interpretation of $D_{n,\ell}$ for $\ell\ge1$. \bigskip

When $\ell\ge 1$, recurrence (\ref{E:Pnlrecurrence}) for $P_{n,\ell}(-1,1)$  gives:
\begin{equation}\label{E:recD}
D_{n,\ell}=D_{n,\ell-1}+D_{n-1,\ell-1}. 
\end{equation}

We say that a subset $\I$ of $[n]$ is \textsl{deranged} by a permutation if no point of $\I$ is fixed by the permutation.
\begin{proposition}
$D_{n,0}=d_n$, the number of derangements of $[n]$. In general, for $\ell\ge 0$, $D_{n,\ell}$ is the number 
of permutations of $[n]$ in which the set $\{1,2,\ldots,n-\ell\}$ is deranged,
with no restrictions on the $\ell$-set $\{n-\ell+1,\ldots,n\}$.
\end{proposition}
\begin{proof}
For $\ell\ge 0$ let $D_{n,\ell}^*$ denote the set 
of permutations in the statement of the Proposition.
 Let $E_{n,\ell}=|D_{n,\ell}^*|$. We claim that $E_{n,\ell}=D_{n,\ell}$. \bigskip

The case $\ell=0$ is immediate. We show that $E_{n,\ell}$ satisfies recurrence \eqref{E:recD}. \medskip

 Now let $\ell>0$. Consider a permutation in $D_{n,\ell}^*$. 
The point $n$ is either (1) deranged,  or (2) not deranged ({\em i.e.}, fixed). \bigskip

\noindent(1)\quad If $n$ is deranged, then the $(n-\ell+1)$-set 
$\{1,2,\ldots,n-\ell,n\}$ is deranged.
By switching $n\leftrightarrow n-\ell+1$ in all permutations of $D_{n,\ell}^*$ 
 we obtain a permutation in $D_{n,\ell-1}^*$. 
Conversely, given any permutation of $D_{n,\ell-1}^*$, 
we switch $n\leftrightarrow n-\ell+1$ to obtain a permutation in $D_{n,\ell}^*$ where  
$n$ is deranged. Hence the number of  permutations in $D_{n,\ell}^*$ with  
$n$ deranged equals $E_{n,\ell-1}$. \bigskip

\noindent(2)\quad 
Here $n$ is fixed so if we remove $n$ from any permutation in $D_{n,\ell}^*$ 
we obtain a permutation in $D_{n-1,\ell-1}^*$. 
Conversely, given a  permutation in $D_{n-1,\ell-1}^*$ we may include $n$ as a fixed point 
to obtain a  permutation in $D_{n,\ell}^*$ with $n$ fixed. 
Hence the number of permutations in $D_{n,\ell}^*$ with $n$  fixed equals $E_{n-1,\ell-1}$. \bigskip

Combining the above two paragraphs shows that $E_{n,\ell}$ satisfies recurrence~(\ref{E:recD}). 
\end{proof}
And a quick check.
$$D_{n,n}=n!$$
there being no restrictions at all in the combinatorial interpretation, in agreement with \eqref{E:Dnell} for $\ell=n$.
\end{subsubsection}\bigskip

\begin{example}
When $n=3$ we have $d_3=D_{3,0}=2$ corresponding to the $2$ permutations of $[3]$ in which $\{1,2,3\}$ is moved: 
$231,312$. \medskip

 Then 
$D_{3,1}=3$ corresponding to the $3$ permutations of $[3]$ in which $\{1,2\}$ is moved: 
$213,231,312$.  \medskip

Then 
$D_{3,2}=4$ corresponding to the $4$ permutations of $[3]$ in which $\{1\}$ is moved: 
$213,231,312,321$. \medskip

Finally 
$D_{3,3}=3!=6$ corresponding to the $3$ permutations of $[3]$ in which $\emptyset$ is moved: 
$123,132,213,231,312,321$. 
\end{example}

Reversing the order of summation in equation \eqref{E:Dnell} gives an alternative expression:
\begin{equation}\label{E:Dnell2}
D_{n,\ell}=\sum_{j=0}^{n-\ell} (-1)^j\,\binom{n-\ell}{j}\,(n-j)!\ .
\end{equation}

\begin{remark}
Formulation \eqref{E:Dnell2} may be proved directly by inclusion-exclusion  on permutations fixing given points.
\end{remark}

\begin{example}
\begin{align*}
D_{5,2}&=\sum_{j=0}^{3} (-1)^j\,\binom{3}{j}\,(5-j)!=\binom{3}{0}5!-\binom{3}{1}4!+\binom{3}{2}3!-\binom{3}{3}2!\\&
=120-72+18-2=64\ .
\end{align*}
\end{example}

Now, from \#2 of Theorem \ref{thm:Pst=}, $s=-1$ and $t=1$ we have:
\begin{equation}\label{E:Dasasum}
D_{n,\ell}=\sum_{j=0}^{\ell}\binom{\ell}{j}d_{n-j}\ .
\end{equation}
Here is a combinatorial explanation. To obtain a permutation in $D_{n,\ell}^*$, we first choose $j$ points from 
$\{n-\ell+1,\ldots,n\}$ to be fixed. Then every derangement of the remaining $(n-j)$ points will produce a permutation in $D_{n,\ell}^*$,
and there are $d_{n-j}$ such derangements.  \bigskip

\begin{example}
\begin{align*}
D_{5,2}&=\sum_{j=0}^{2} \binom{2}{j}\,d_{5-j}=
\binom{2}{0}d_{5}+\binom{2}{1}d_{4}+\binom{2}{2}d_{3}\\&=1\times 44+2\times 9+1\times 2=44+18+2=64\ .
\end{align*}
\end{example}

\begin{subsubsection}{Permanents from $J-I$}
Theorem \ref{thm:xnl} specializes to
$$(J-I)^{\vee\ell}=\sum_{k=0}^{\min(\ell,n-\ell)} D_{\ell, k}\, \JS{k}{n\ell}\ .$$
This can be written using the hypergeometric form:
$$(J-I)^{\vee\ell}=\sum_{k=0}^{\min(\ell,n-\ell)} \, 
(-1)^{\ell-k}\,k!\,{}_2 F_0\left(\genfrac{}{}{0pt}{}{ k-\ell ,1+k}{\text{---}}\biggm| 1 \right)\,\JS{k}{n\ell}\ .$$
with spectrum 
\begin{align*}
&\text{eigenvalue } (-1)^\ell\,{}_2 F_0\left(\genfrac{}{}{0pt}{}{ \alpha-\ell ,-\alpha+n-\ell+1}{\text{---}}\biggm| 1 \right) \\
&\text{ occurring with multiplicity } \binom{n}{\alpha}-\binom{n}{\alpha-1}
\end{align*}
by Corollary \ref{cor:spec} and Proposition \ref{prop:spec}. \bigskip

The entries of $(J-I)^{\vee\ell}$ are from the set of numbers $D_{n,\ell}$. For the spectrum, start with $\alpha=0$.
From equation \eqref{E:Dnell}, we have
$$(-1)^\ell\,{}_2 F_0\left(\genfrac{}{}{0pt}{}{ -\ell ,n-\ell+1}{\text{---}}\biggm| 1 \right)=\frac{1}{(n-\ell)!}\,D_{n,n-\ell}$$
As $\alpha$ increases, we see that the spectrum consists of the numbers
$$\frac{(-1)^\alpha}{(n-\ell-\alpha)!}\,D_{n-2\alpha,n-\ell-\alpha}$$
Think of moving in the derangement triangle, as in the Appendix, starting from position $n,n-\ell$, rescaling the values by
the factorial of the column at each step. Then the eigenvalues are found by successive knight's moves, up 2 rows and one column to the left,
with alternating signs.

\begin{example}
For $n=5$, $\ell=3$, we have
$$(J-I)^{\vee 3}= 
 \left[ \begin {array}{cccccccccc} 2&3&3&3&3&4&3&3&4&4
\\\noalign{\medskip}3&2&3&3&4&3&3&4&3&4\\\noalign{\medskip}3&3&2&4&3&3
&4&3&3&4\\\noalign{\medskip}3&3&4&2&3&3&3&4&4&3\\\noalign{\medskip}3&4
&3&3&2&3&4&3&4&3\\\noalign{\medskip}4&3&3&3&3&2&4&4&3&3
\\\noalign{\medskip}3&3&4&3&4&4&2&3&3&3\\\noalign{\medskip}3&4&3&4&3&4
&3&2&3&3\\\noalign{\medskip}4&3&3&4&4&3&3&3&2&3\\\noalign{\medskip}4&4
&4&3&3&3&3&3&3&2\end {array} \right]$$
with characteristic polynomial
$${\lambda}^{5} \left( \lambda-32 \right)  \left( \lambda+3 \right) ^{4}$$

\end{example}

\begin{remark}
Except for $\ell=2$, the coefficients in the expansion of $(J-I)^{\vee\ell}$ in the Johnson basis will be distinct. Thus 
the Johnson basis itself can be read off directly from $(J-I)^{\vee\ell}$. In this sense, the centralizer algebra of the action of
the symmetric group on $\ell$-sets is determined by knowledge of the action of just $J-I$ on $\ell$-sets.
\end{remark}

\end{subsubsection}


\subsection{Generalized arrangements of $[n]$}\label{SS:gga}

Given $[n]$, $0\le j\le n$, a $j${\em-arrangement\/} of $[n]$ is a permutation of a $j$-subset of $[n]$.
The number of $j$-arrangements of $[n]$ is
$$
A(n,j)=\frac{n!}{(n-j)!}\ .
$$
Note that there is a single $0$-arrangement of $[n]$, from the empty set. \bigskip

Define $A_{n,\ell}=P_{n,\ell}(1,1)$.
So, similar to the case for derangements, 
equation \eqref{E:Pst1} gives:
\begin{equation}\label{E:Anell=}
A_{n,\ell}=\sum_{j=\ell}^{n}\binom{n-\ell}{n-j}{j!}=\ell!\,{}_2 F_0\left(\genfrac{}{}{0pt}{}{ \ell-n ,1+\ell}{\text{---}}\biggm| -1 \right)\ .
\end{equation}
Now define
$a_n=A_{n,0}$ so $$a_n=\per(I+J)=\sum_{j=0}^{n}\frac{n!}{(n-j)!}=\sum_{j=0}^{n}A(n,j)$$
 is the {\em total} number of $j$-arrangements of $[n]$ for $j=0,1,\ldots,n$. 
Thus we have a combinatorial interpretation of $A_{n,0}$.

\begin{subsubsection}{Combinatorial interpretation of $A_{n,\ell}$}
We now give a combinatorial interpretation of $A_{n,\ell}$ for $\ell\ge 1$. \bigskip

When $\ell\ge 1$, recurrence \eqref{E:Pnlrecurrence} for $P_{n,\ell}(1,1)$ gives: 
\begin{equation}\label{E:recA}
A_{n,\ell}=A_{n,\ell-1}-A_{n-1,\ell-1}. 
\end{equation}

\begin{proposition}
$A_{n,0}=a_n$, the total number of arrangements of $[n]$. In general,  for $\ell\ge 0$, 
$A_{n,\ell}$ is  the number of arrangements of $[n]$ which contain $\{1,2,\ldots,\ell\}$. 
\end{proposition}
\begin{proof}
For $\ell\ge 0$, let $A_{n,\ell}^*$ denote the set of arrangements of $[n]$ which contain $[\ell]$.
With $[0]=\emptyset$, we note that $A_{n,0}^*$ is the set of all arrangements.
Let $B_{n,\ell}=|A_{n,\ell}^*|$.  
We claim that $B_{n,\ell}=A_{n,\ell}$. \bigskip

The initial values with $\ell=0$ are immediate. We show  that $B_{n,\ell}$ satisfies recurrence \eqref{E:recA}. \medskip

Consider $A_{n,\ell-1}^*$. 
Let $A\in A_{n,\ell-1}^*$, so $A$ is an arrangement of $[n]$ containing 
$[\ell-1]$. If $\ell=1$, then $A\in A_{n,0}^*$ is any arrangement.
Now either $\ell\in A$ or $\ell\not\in A$.  \medskip

If $\ell\in A$, then $A\in A_{n,\ell}^*$, 
and so the number of arrangements in $A_{n,\ell-1}^*$ which contain  $\ell$ equals 
$B_{n,\ell}$.  \medskip

If $\ell\not\in A$, 
then by subtracting $1$ from all parts of $A$ which are $\ge \ell+1$ 
we obtain an arrangement of $[n-1]$ which contains $[\ell-1]$,
i.e., an arrangement in $A_{n-1,\ell-1}^*$. 
Conversely, given an arrangement in $A_{n-1,\ell-1}^*$, adding $1$ to all parts 
$\ge \ell$ yields an arrangement in $A_{n,\ell-1}^*$ which does not contain $\ell$. 
Hence the number of arrangements in $A_{n,\ell-1}^*$ which do not contain $\ell$ equals $B_{n-1,\ell-1}$. \medskip

We conclude that $B_{n,\ell-1}=B_{n,\ell}+B_{n-1,\ell-1}$, hence the result.
\end{proof}
\end{subsubsection}\bigskip

\begin{example}
When $n=3$ we have $a_3=A_{3,0}=16$ corresponding to the $16$ arrangements of $[3]$: 
$[\,],1,2,3,12,21,13,31,23,32,123,132,213,231,312,321$. \medskip

 Then 
$A_{3,1}=11$ corresponding to the $11$ arrangements of $[3]$ which contain $\{1\}$: 
$1,12,21,13,31,123,132,213,231,312,321$. \medskip

 Then 
$A_{3,2}=8$ corresponding to the $8$ arrangements of $[3]$ which contain $\{1,2\}$: 
$12,21,123,132,213,231,312,321$. \medskip

Finally,  
$A_{3,3}=3!=6$ corresponding to the $6$ arrangements of $[3]$ which contain $\{1,2,3\}$: 
$123,132,213,231,312,321$. 
\end{example}

Rearranging the factors in equation \eqref{E:Pst1}, we have
$$
P_{n,\ell}(s,t)=\sum_{j=\ell}^nA(j,\ell)\,A(n-\ell,j-\ell)\,s^{n-j}t^j\ . 
$$
With $s=t=1$ this gives:
\begin{equation}\label{eq:anl}
A_{n,\ell}=\sum_{j=\ell}^nA(j,\ell)\,A(n-\ell,j-\ell)\ .
\end{equation}

Here is a combinatorial explanation:\medskip

For any $j\ge \ell$, to obtain a $j$-arrangement $A$ of $[n]$ containing $[\ell]$ 
we may place the $\ell$ points of $\{1,2,\ldots,\ell\}$ into these $j$ positions in 
$A(j,\ell)$ ways. 
Then the remaining $(j-\ell)$ positions in $A$ can be filled in by a $(j-\ell)$-arrangement of the unused 
$(n-\ell)$ points in $A(n-\ell,j-\ell)$ ways. Hence \eqref{eq:anl}.\bigskip

\begin{example}
\begin{align*}
A_{5,2}&=\sum_{j=2}^5A(j,2)\,A(3,j-2)\\&=
A(2,2)\,A(3,0)+A(3,2)\,A(3,1)+A(4,2)\,A(3,2)+A(5,2)\,A(3,3)\\
&=2\times1+6\times3+12\times6+20\times6\\&=2+18+72+120=212\ .
\end{align*}
\end{example}

Finally, from \#2 of Theorem \ref{thm:Pst=}, $s=1$ and $t=1$ we have:
\begin{equation*}\label{E:Aasasum}
A_{n,\ell}=\sum_{j=0}^{\ell}(-1)^j\binom{\ell}{j}a_{n-j}\ .
\end{equation*}

\begin{example}
\begin{align*}
A_{5,2}&=\sum_{j=0}^{2}(-1)^j\binom{2}{j}a_{5-j}
=\binom{2}{0}a_5-\binom{2}{1}a_4+\binom{2}{2}a_3\\
&=1\times 326-2\times 65+1\times 16=326-130+16=212\ .
\end{align*}
\end{example}

\begin{subsubsection}{Permanents from $I+J$}
Theorem \ref{thm:xnl} specializes to
$$(I+J)^{\vee\ell}=\sum_{k=0}^{\min(\ell,n-\ell)} A_{\ell, k}\, \JS{k}{n\ell}\ .$$
This can be written using the hypergeometric form:
$$(I+J)^{\vee\ell}=\sum_{k=0}^{\min(\ell,n-\ell)} \, 
k!\,{}_2 F_0\left(\genfrac{}{}{0pt}{}{ k-\ell ,1+k}{\text{---}}\biggm| -1 \right)\,\JS{k}{n\ell}\ .$$
with spectrum 
\begin{align*}
&\text{eigenvalue } {}_2 F_0\left(\genfrac{}{}{0pt}{}{ \alpha-\ell ,-\alpha+n-\ell+1}{\text{---}}\biggm| -1 \right) \\
&\text{ occurring with multiplicity } \binom{n}{\alpha}-\binom{n}{\alpha-1}
\end{align*}
by Corollary \ref{cor:spec} and Proposition \ref{prop:spec}. \bigskip

\begin{example}
For $n=5$, $\ell=3$, we have
$$(I+J)^{\vee 3}= 
  \left[ \begin {array}{cccccccccc} 16&11&11&11&11&8&11&11&8&8
\\\noalign{\medskip}11&16&11&11&8&11&11&8&11&8\\\noalign{\medskip}11&
11&16&8&11&11&8&11&11&8\\\noalign{\medskip}11&11&8&16&11&11&11&8&8&11
\\\noalign{\medskip}11&8&11&11&16&11&8&11&8&11\\\noalign{\medskip}8&11
&11&11&11&16&8&8&11&11\\\noalign{\medskip}11&11&8&11&8&8&16&11&11&11
\\\noalign{\medskip}11&8&11&8&11&8&11&16&11&11\\\noalign{\medskip}8&11
&11&8&8&11&11&11&16&11\\\noalign{\medskip}8&8&8&11&11&11&11&11&11&16
\end {array} \right]$$
with characteristic polynomial
$$\left( \lambda-106 \right)  \left( \lambda-11 \right) ^{4} \left( 
\lambda-2 \right) ^{5}\ .$$
\end{example}

As for the case of derangements, the Johnson basis can be read off directly from the matrix $(I+J)^{\vee\ell}$.

\end{subsubsection}

\hfill \bigskip

\textbf{Acknowledgment.} We would like to thank Stacey Staples for discussions about zeons and trace formulas. 

\vfill \eject
\section{Appendix}

\subsection*{Generalized derangement numbers and integer sequences}
The first two columns of the $D_{n,\ell}$ triangle, 
$D_{n,0}$ and $D_{n,1}$, 
give sequences A000166 and A000255 in the 
On-Line Encyclopedia of Integer Sequences ~\cite{OEIS}. 
The comments for A000255 do not contain our combinatorial interpretation. \bigskip

The first two columns of the $A_{n,\ell}$ triangle, 
$A_{n,0}$ and $A_{n,1}$, 
give sequences A000522 and A001339. 
The comments contain our combinatorial interpretation. 
The next two columns, 
$A_{n,2}$ and $A_{n,3}$, 
gives sequences A001340 and A00134; 
here our combinatorial interpretation is not mentioned in the comments. 
\vfill\eject

 Generalized Derangement Triangles\\[.1in]
\thispagestyle{empty}
$\ell=0$ is the leftmost column. \\
The rows correspond to $n$ from $0$ to $9$. \bigskip


Values of $D_{n,\ell}$. \bigskip

$$
 \left[ \begin {array}{cccccccccc} 1&0&0&0&0&0&0&0&0&0
\\\noalign{\medskip}0&1&0&0&0&0&0&0&0&0\\\noalign{\medskip}1&1&2&0&0&0
&0&0&0&0\\\noalign{\medskip}2&3&4&6&0&0&0&0&0&0\\\noalign{\medskip}9&
11&14&18&24&0&0&0&0&0\\\noalign{\medskip}44&53&64&78&96&120&0&0&0&0
\\\noalign{\medskip}265&309&362&426&504&600&720&0&0&0
\\\noalign{\medskip}1854&2119&2428&2790&3216&3720&4320&5040&0&0
\\\noalign{\medskip}14833&16687&18806&21234&24024&27240&30960&35280&
40320&0\\\noalign{\medskip}133496&148329&165016&183822&205056&229080&
256320&287280&322560&362880\end {array} \right] 
$$ \\[.1in]

Values of $A_{n,\ell}$.\bigskip
$$ 
 \left[ \begin {array}{cccccccccc} 1&0&0&0&0&0&0&0&0&0
\\\noalign{\medskip}2&1&0&0&0&0&0&0&0&0\\\noalign{\medskip}5&3&2&0&0&0
&0&0&0&0\\\noalign{\medskip}16&11&8&6&0&0&0&0&0&0\\\noalign{\medskip}
65&49&38&30&24&0&0&0&0&0\\\noalign{\medskip}326&261&212&174&144&120&0&0
&0&0\\\noalign{\medskip}1957&1631&1370&1158&984&840&720&0&0&0
\\\noalign{\medskip}13700&11743&10112&8742&7584&6600&5760&5040&0&0
\\\noalign{\medskip}109601&95901&84158&74046&65304&57720&51120&45360&
40320&0\\\noalign{\medskip}986410&876809&780908&696750&622704&557400&
499680&448560&403200&362880\end {array} \right] 
$$
\vfill \eject
Exponential polynomials $h_{n,m}(s,t)$.  Note that, as is common for matrix indexing, we have dropped the commas in the numerical subscripts.\bigskip

$n=0$\par
$$h_{00}=1 \,,\qquad h_{01}=t \,,\qquad h_{02}=2t^2 \,,\qquad h_{03}=6t^3 \,,\qquad h_{04}=24t^4$$
\hfill\medskip

$n=1$ \par
$$h_{10}=s+t \,,\quad h_{11}=st+2t^2\,,\quad h_{12}=2st^2+6t^3 \,,\quad h_{13}=6st^3+24t^4 \,,\quad h_{14}=24st^4+120t^5$$
\hfill\medskip

$n=2$\par
\begin{align*}
h_{20}&=s^2+2st+2t^2 \,,\qquad h_{21}=s^2t+4st^2+6t^3 \,,\qquad h_{22}=2s^2t^2+12st^3+24t^4 \\ 
h_{23}&=6s^2t^3+48st^4+120t^5 \,,\qquad h_{24}=24s^2t^4+240st^5+720t^6
\end{align*}
\hfill\medskip

$n=3$\par
\begin{align*}
h_{30}&=s^3+3s^2t+6st^2+6t^3 \\  h_{31}&=s^3t+6s^2t^2+18st^3+24t^4 \,,\quad h_{32}=2s^3t^2+18s^2t^3+72st^4+120t^5 \\
h_{33}&=6s^3t^3+72s^2t^4+360st^5+720t^6 \,,\qquad h_{34}=24s^3t^4+360s^2t^5+2160st^6+5040t^7
\end{align*}
\hfill\medskip

$n=4$\par
\begin{align*}
h_{40}&=s^4+4s^3t+12s^2t^2+24st^3+24t^4 \,, \qquad h_{41}=s^4t+8s^3t^2+36s^2t^3+96st^4+120t^5\\
h_{42}&=2s^4t^2+24s^3t^3+144s^2t^4+480st^5+720t^6 \\ 
h_{43}&=6s^4t^3+96s^3t^4+720s^2t^5+2880st^6+5040t^7\\ 
h_{44}&=24s^4t^4+480s^3t^5+4320s^2t^6+20160st^7+40320t^8
\end{align*}
\vfill\eject
{
\vglue.525in
\setlength{\unitlength}{0.0575cm}
\begin{center}
\begin{picture}(180,50)(0,-30)
\put(-15,34){{
\small$
\begin{bmatrix}
s+t&t&t\\
t&s+t&t\\
t&t&s+t\\
\end{bmatrix}$}}

\put(75,43){\circle*{3}}
\put(66,25){\circle*{3}}
\put(84,25){\circle*{3}}
\put(66,25){\line(1,0){18}}
\put(66,25){\line(1,2){9}}
\put(75,43){\line(1,-2){9}}
\put(68.5,46){$s+t$}
\put(82,33){$t$}
\put(82,19){$s+t$}
\put(74,20){$t$}
\put(56,19){$s+t$}
\put(66.5,33){$t$}

\put(129,34){$P_{3,0}=s^3+3s^2t+6st^2+6t^3$}

\put(-40,45){$\ell=0$}


\put(10,5){\circle*{3}}
\put(1,-13){\circle*{3}}
\put(19,-13){\circle*{3}}
\put(1,-23){$(s+t)^3$}
%
\put(45,5){\circle*{3}}
\put(36,-13){\circle*{3}}
\put(54,-13){\circle*{3}}
\put(45,5){\line(1,-2){9}}
\put(35,-23){$(s+t)t^2$}
%
\put(80,5){\circle*{3}}
\put(71,-13){\circle*{3}}
\put(89,-13){\circle*{3}}
\put(71,-13){\line(1,0){18}}
\put(69,-23){$(s+t)t^2$}
%
\put(115,5){\circle*{3}}
\put(106,-13){\circle*{3}}
\put(124,-13){\circle*{3}}
\put(106,-13){\line(1,2){9}}
\put(105,-23){$(s+t)t^2$}
%
\put(150,5){\circle*{3}}
\put(141,-13){\circle*{3}}
\put(159,-13){\circle*{3}}
\put(141,-13){\line(1,0){18}}
\put(141,-13){\line(1,2){9}}
\put(150,5){\line(1,-2){9}}
\put(149,-23){$t^3$}

\put(-10,-56){{
\small$
\begin{bmatrix}
s+t&t&t\\
t&s+t&t\\
t&t&t\\
\end{bmatrix}$}}

\put(75,-47){\circle*{3}}
\put(66,-65){\circle{3}}
\put(84,-65){\circle*{3}}
\put(67.5,-65){\line(1,0){18}}
\put(67,-64){\line(1,2){9}}
\put(75,-47){\line(1,-2){9}}
\put(68.5,-44){$s+t$}
\put(82,-57){$t$}
\put(82,-71){$s+t$}
\put(74,-70){$t$}
\put(62,-71){$t$}%
\put(66.5,-57){$t$}

\put(134,-56){$P_{3,1}=s^2t+4st^2+6t^3$}

\put(-40,-45){$\ell=1$}


\put(10,-85){\circle*{3}}
\put(1,-103){\circle{3}}
\put(19,-103){\circle*{3}}
\put(-1,-113){$(s+t)^2t$}
%
\put(45,-85){\circle*{3}}
\put(36,-103){\circle{3}}
\put(54,-103){\circle*{3}}
\put(45,-85){\line(1,-2){9}}
\put(44,-113){$t^3$}
%
\put(80,-85){\circle*{3}}
\put(71,-103){\circle{3}}
\put(89,-103){\circle*{3}}
\put(72.5,-103){\line(1,0){18}}
\put(69,-113){$(s+t)t^2$}
%
\put(115,-85){\circle*{3}}
\put(106,-103){\circle{3}}
\put(124,-103){\circle*{3}}
\put(107,-102){\line(1,2){9}}
\put(105,-113){$(s+t)t^2$}
%
\put(150,-85){\circle*{3}}
\put(141,-103){\circle{3}}
\put(159,-103){\circle*{3}}
\put(142.5,-103){\line(1,0){18}}
\put(142,-102){\line(1,2){9}}
\put(150,-85){\line(1,-2){9}}
\put(149,-113){$t^3$}
\put(-5,-146){{
\small$
\begin{bmatrix}
s+t&t&t\\
t&t&t\\
t&t&t\\
\end{bmatrix}$}}

\put(75.2,-137){\circle*{3}}
\put(66,-155){\circle{3}}
\put(84,-155){\circle{3}}
\put(67.5,-155){\line(1,0){15}}
\put(67,-154){\line(1,2){9}}
\put(75,-137){\line(1,-2){8.2}}
\put(68.5,-134){$s+t$}
\put(82,-147){$t$}
\put(82,-161){$t$}
\put(74,-160){$t$}
\put(62,-161){$t$}%
\put(66.5,-147){$t$}

\put(139,-146){$P_{3,2}=2st^2+6t^3$}

\put(-40,-135){$\ell=2$}


\put(10,-175){\circle*{3}}
\put(1,-193){\circle{3}}
\put(19,-193){\circle{3}}
\put(-1,-203){$(s+t)t^2$}
%
\put(45,-175){\circle*{3}}
\put(36,-193){\circle{3}}
\put(54,-193){\circle{3}}
\put(45,-175){\line(1,-2){8.2}}
\put(44,-203){$t^3$}
%
\put(80,-175){\circle*{3}}
\put(71,-193){\circle{3}}
\put(89,-193){\circle{3}}
\put(72.2,-193){\line(1,0){15.5}}
\put(69,-203){$(s+t)t^2$}
%
\put(115,-175){\circle*{3}}
\put(106,-193){\circle{3}}
\put(124,-193){\circle{3}}
\put(107,-192){\line(1,2){9}}
\put(114,-203){$t^3$}
%
\put(150,-175){\circle*{3}}
\put(141,-193){\circle{3}}
\put(159,-193){\circle{3}}
\put(142.5,-193){\line(1,0){15.0}}
\put(142,-192){\line(1,2){9}}
\put(150,-175){\line(1,-2){8.2}}
\put(149,-203){$t^3$}

\put(0,-236){{
\small$
\begin{bmatrix}
t&t&t\\
t&t&t\\
t&t&t\\
\end{bmatrix}$}}

\put(75,-227){\circle{3}}
\put(66,-245){\circle{3}}
\put(84,-245){\circle{3}}
\put(67.2,-245){\line(1,0){15.5}}
\put(67,-244){\line(1,2){7.8}}
\put(75.5,-228.5){\line(1,-2){7.8}}
\put(74,-224){$t$}
\put(82,-237){$t$}
\put(82,-251){$t$}
\put(74,-250){$t$}
\put(62,-251){$t$}%
\put(66.5,-237){$t$}

\put(144,-236){$P_{3,3}=6t^3$}

\put(-40,-225){$\ell=3$}


\put(10,-265){\circle{3}}
\put(1,-283){\circle{3}}
\put(19,-283){\circle{3}}
\put(8,-293){$t^3$}
%
\put(44.5,-265){\circle{3}}
\put(36,-283){\circle{3}}
\put(53,-283){\circle{3}}
\put(45,-266.6){\line(1,-2){7.6}}
\put(44,-293){$t^3$}
%
\put(80,-265){\circle{3}}
\put(71,-283){\circle{3}}
\put(89,-283){\circle{3}}
\put(72.2,-283){\line(1,0){15.6}}
\put(79,-293){$t^3$}
%
\put(115,-265){\circle{3}}
\put(106,-283){\circle{3}}
\put(124,-283){\circle{3}}
\put(107,-282){\line(1,2){7.7}}
\put(114,-293){$t^3$}
%
\put(150,-265){\circle{3}}
\put(141,-283){\circle{3}}
\put(159,-283){\circle{3}}
\put(142.3,-283){\line(1,0){15.5}}
\put(142,-282){\line(1,2){7.6}}
\put(150,-266.5){\line(1,-2){7.8}}
\put(149,-293){$t^3$}

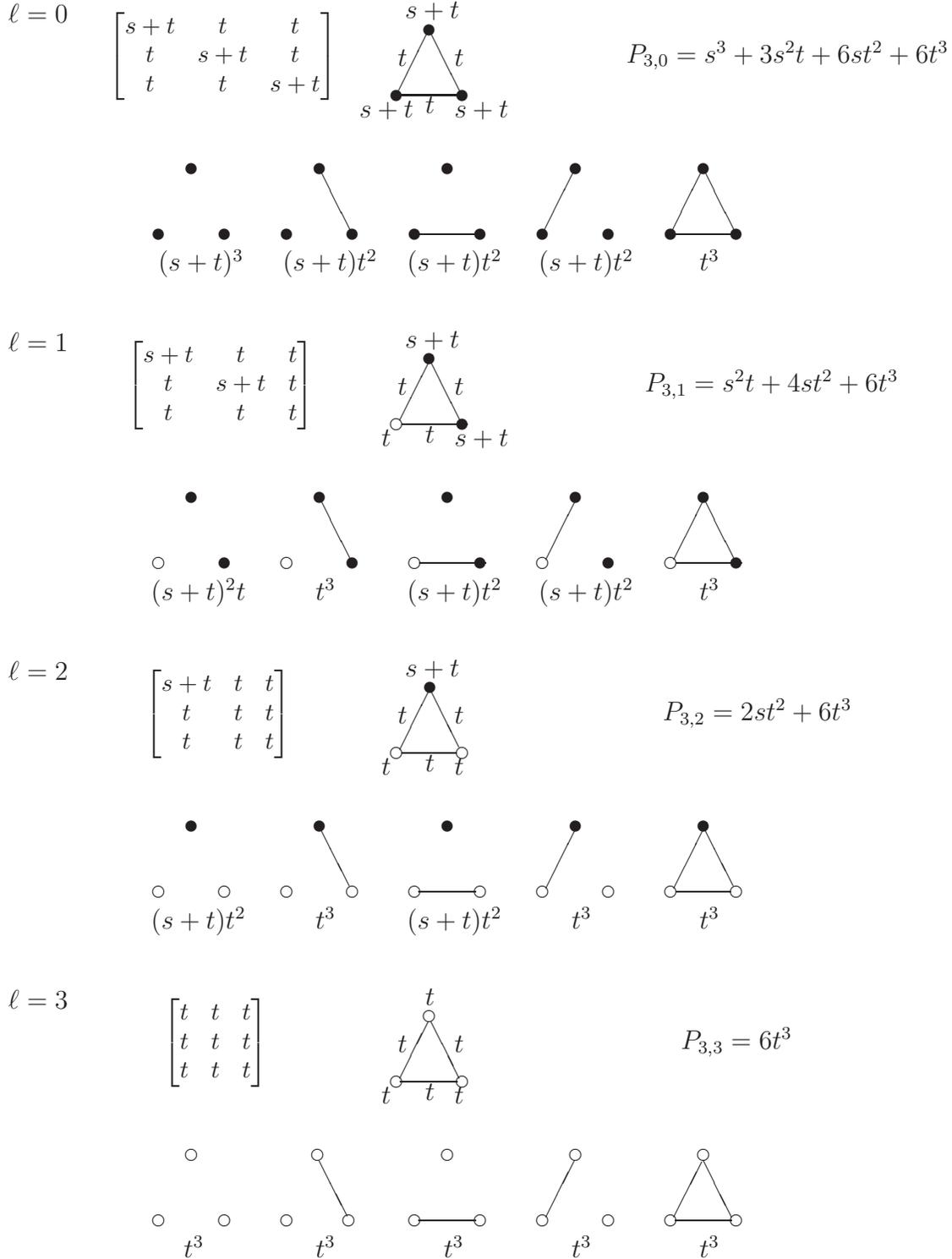
\put(-40,-310){FIGURE 2. $M_{3,\ell}, K_3^{(\ell)}, P_{3,\ell}$ and the 5 weighted elementary subgraphs
of $K_3^{(\ell)}$ for $\ell=0,1,2$, $3$.}
\put(30,-320){Distinguished vertices are shown {\bf bold}.}
\end{picture}
\end{center}
}
\vfill\eject

\end{document}